\documentclass[12pt]{amsart}

\usepackage[
text={440pt,575pt},
headheight=9pt,
centering
]{geometry}

\usepackage{amsmath, amssymb, amsthm, enumerate, bm}
\allowdisplaybreaks 
\usepackage{mathtools}

\usepackage{graphicx}
\usepackage{subcaption}
\usepackage[normalem]{ulem}

\usepackage{multirow}

\usepackage{pifont}
\usepackage{mathptmx}

\usepackage{xcolor}
\usepackage{hyperref}
\definecolor{darkblue}{RGB}{0,0,160}
\hypersetup{
colorlinks,%
citecolor=darkblue,%
filecolor=black,%
linkcolor=darkblue,%
urlcolor=darkblue
}

\usepackage{todonotes}
\makeatletter
\newcommand{\nolisttopbreak}{\vspace{\topsep}\nobreak\@afterheading}
\makeatother

\theoremstyle{definition}
\newtheorem{definition}{Definition}[section]

\newtheorem{theorem}[definition]{Theorem}
\newtheorem{proposition}[definition]{Proposition}
\newtheorem{lemma}[definition]{Lemma}
\newtheorem{corollary}[definition]{Corollary}

\newtheorem*{fact*}{Fact}
\newtheorem{remark}[definition]{Remark}
\newtheorem{example}[definition]{Example}
\newtheorem{question}[definition]{Question}

\newtheorem{problem}[definition]{Problem}
\newtheorem{conjecture}[definition]{Conjecture}

\newcommand{\C}{\mathcal{C}}

\newcommand{\Cf}{\mathfrak{C}}

\newcommand{\Hc}{\mathcal{H}}

\newcommand{\M}{\mathcal{M}}
\newcommand{\NN}{\mathbb{N}}

\newcommand{\RR}{\mathbb{R}}

\newcommand{\ZZ}{\mathbb{Z}}
\newcommand{\QQ}{\mathbb{Q}}

\newcommand\nub{{\boldsymbol 0}}
\newcommand\ab{{\boldsymbol a}}
\newcommand\bb{{\boldsymbol b}}
\newcommand\eb{{\boldsymbol e}}
\newcommand\rb{{\boldsymbol r}}
\newcommand\ub{{\boldsymbol u}}
\newcommand\vb{{\boldsymbol v}}
\newcommand\wb{{\boldsymbol w}}

\DeclareMathOperator{\cncl}{cn}
\DeclareMathOperator{\mndcl}{mn}

\DeclareMathOperator{\Sym}{Sym}

\DeclareMathOperator{\supp}{supp}

\DeclareMathOperator{\width}{width}

\DeclareMathOperator{\Span}{span}

\DeclareMathOperator{\ind}{ind}

\newcommand\defas{\coloneqq}

\makeatletter
\@namedef{subjclassname@2020}{%
  \textup{2020} Mathematics Subject Classification}
\makeatother

\title[Theorems of Carath\'{e}odory, Minkowski-Weyl, and Gordan
up to symmetry]{
Theorems of Carath\'{e}odory, Minkowski-Weyl, and Gordan\\
up to symmetry
}

\author{Dinh Van Le}
\address{Universit\"at Osnabr\"uck\\ Osnabr\"uck, Germany}
\email{dlevan@uos.de}

\author{Tim R\"omer}
\address{Universit\"at Osnabr\"uck\\ Osnabr\"uck, Germany}
\email{troemer@uos.de}



\begin{document}
\begin{abstract}
In this paper we extend three classical and fundamental results in polyhedral geometry, namely, Carath\'{e}odory's theorem, the  Minkowski-Weyl theorem, and Gordan's lemma to infinite dimensional spaces, in which considered cones and monoids are invariant under actions of symmetric groups.
\end{abstract}
\maketitle


\section{Introduction}
\label{sec:introduction}

This paper is motivated by the desire to extend classical and fundamental results in polyhedral geometry to the \emph{infinite dimensional} case. Of course, without further constraints, one cannot hope for any feasible and meaningful extensions. Therefore, one has first to clarify what  \emph{reasonable} constraints could be.
A possible answer to this question is given
in our joint work  with Kahle \cite{KLR}, where a general framework is developed to study sets  in infinite dimensional ambient spaces
that are invariant under the action of the infinite symmetric group $\Sym$ (see Section~\ref{section-cones} for precise definitions). The main results of \cite{KLR} allow to characterize when objects of interest in polyhedral geometry, such as cones, monoids, and polytopes, are equivariantly finitely generated. It is thus a natural goal to extend polyhedral geometry to this equivariant setting, for which the following problem is of fundamental importance:

\begin{problem}
 \label{pb:main}
Extend classical results in polyhedral geometry to infinite dimensional spaces, in which considered objects are invariant under actions of symmetric groups.
\end{problem}

It should be noted that this problem and also the work \cite{KLR} are inspired by the rapidly growing theory of symmetric ideals in infinite dimensional polynomial rings. Such ideals appeared naturally in various areas of mathematics and have recently been intensively studied. See \cite{Dr14} for an excellent survey. One prominent theme in this theory is to find equivariant versions of known results related to ideals in Noetherian polynomial rings. Examples of successful extensions include equivariant Hilbert's basis theorem \cite{AH07,Co67,Co87,HS12,NR19}, equivariant Hilbert-Serre theorem \cite{KLS,Na,NR17}, equivariant Buchberger algorithm \cite{Co87,HKL}, equivariant Hochster's formula \cite{MR20}. See, e.g., also \cite{DEF,LNNR,LNNR2,LR20,NS,SS16,SaSn17} for related results.

In the present work, as an initiation of the study of Problem~\ref{pb:main}, we search for equivariant versions of three classical results: Carath\'{e}odory's theorem, the Minkowski-Weyl theorem, and Gordan's lemma.

More specifically, the following questions and conjecture are considered:

\begin{question}
\label{q:Caratheodory}
 Given a $\Sym$-invariant cone $C\subseteq\RR^\infty$ with generating set $A$ and $\ub\in C$, what is the bound for the number of elements of $A$ needed to generate $\ub$? Is there a uniform bound for all $\ub\in C$?
\end{question}

\begin{question}
\label{q:Weyl-Minkowski}
 Let $C\subseteq\RR^\infty$ be a $\Sym$-equivariantly finitely generated cone. Is the dual cone $C^*$ of $C$ also $\Sym$-equivariantly finitely generated?
\end{question}

\begin{conjecture}
 \label{cj:Gordan}
 Let $C\subseteq\RR^\infty$ be a $\Sym$-equivariantly finitely generated rational cone. Then $M=C\cap\ZZ^\infty$ is a $\Sym$-equivariantly finitely generated normal monoid.
\end{conjecture}

Note that a local version of Conjecture~\ref{cj:Gordan} (stated for a chain of finite dimensional cones) has been studied in a special case by Ananiadi in her thesis (see \cite[Conjecture 3.4.4]{An}). There the conjecture is verified  for a chain of cones that is $\Sym$-equivariantly generated by one element of $\ZZ^2_{\ge0}$ (see \cite[Theorem 3.4.6]{An}).

Our general approach to the above questions is to employ the framework developed in \cite{KLR}. More precisely, to any \emph{global} cone $C\subseteq\RR^\infty$ one can associate a chain of \emph{local} cones $\C=(C_n)_{n\ge1}$ with $C_n=C\cap\RR^n$ for $n\ge1$. One of the main results of \cite{KLR} is a \emph{local-global principle} (see \cite[Theorem 4.10]{KLR}) that allows to characterize certain properties of $C$ in terms of the chain $\C$ (see \cite[Corollary 5.4]{KLR}). Thus, one can study a global cone $C$ via the local cones $C_n$ in finite dimensional vector spaces and vice versa. A similar idea also applies to monoids in $\ZZ^\infty$.

Let us now outline the main results of the paper. Regarding Question~\ref{q:Caratheodory}, we show that if $A\subseteq\RR^\infty_{\ge0}$, then the number of elements of $A$ needed to generate $\ub\in C$ is bounded by the support size of $\ub$ (Theorem~\ref{thm-Caratheodory}). This bound depends on $\ub$ and there is no uniform bound for every element of $C$ (see Example~\ref{ex:no-bound}). Somewhat unexpectedly, Question~\ref{q:Weyl-Minkowski} admits a negative answer: we show that the dual $C^*$ of a cone $C\subseteq\RR^\infty_{\ge0}$ is too large to be  $\Sym$-equivariantly finitely generated (Theorem~\ref{thm:global-dual}). As a remedy for this, we consider the chain $\C=(C_n)_{n\ge1}$ of local cones and give a complete description of the chain $\C^*=(C_n^*)_{n\ge1}$ of dual cones (Theorem~\ref{t:Minkowski-Weyl}). Finally, Conjecture~\ref{cj:Gordan} is verified for nonnegative cones $C\subseteq\RR^\infty_{\ge0}$ (Theorem~\ref{t:Gordan}). The main ingredients in the proof of this result are new characterizations of related properties of chains of cones and monoids (Lemmas~\ref{l:globallocalmonoids}, \ref{l:stabilizing-cone}, \ref{l:stabilizing-monoid}). One might have noticed that nonnegativity is a vital assumption for our results. The main reason for this restriction is that nonnegativity is necessary to apply the local-global principle \cite[Theorem 4.10]{KLR} to cones and monoids (see \cite[Section 5]{KLR} and also Sections~\ref{section-cones}, \ref{sec:monoid} for more details).

The remaining parts of the paper are organized as follows. The next section introduces the needed notions and discusses the relationship between a global cone and the corresponding chain of local cones. The answers to Questions~\ref{q:Caratheodory} and \ref{q:Weyl-Minkowski} are given in Sections~\ref{sec-Caratheodory} and \ref{sec-Minkowski-Weyl}, respectively. Analogously to Section~\ref{section-cones}, we consider in Section~\ref{sec:monoid} the relationship between a global (normal) monoid and its chain of local monoids. Section~\ref{section-Gordan} contains the proof of Conjecture~\ref{cj:Gordan} for nonnegative cones and its consequences. We close the paper with a brief discussion of applications of our results in Section~\ref{sec-Discussion}.

\section{Chains of cones}
\label{section-cones}

In this section we briefly review the notion of $\Sym$-invariant chains of cones introduced in \cite{KLR}. Also the relationship between a chain of cones and the associated global cone is discussed. But let us start with symmetric groups and their actions.

\subsection{Symmetric groups}
\label{sec:Sym}

With $\NN = \{1,2,\dots \}$ being the set of positive integers, let $\Sym(\NN)$ denote the symmetric group on $\NN$, i.e. the group of all permutations of $\NN$.
For any $n\in \NN$ let $\Sym(n)$ be the symmetric group on $[n]=\{1,\dots,n\}$.  Regarding $\Sym(n)$ as the subgroup of $\Sym(\NN)$ that fixes every $k\ge n+1$, one obtains an increasing
chain of finite subgroups
$$
  \Sym(1)\subset \Sym(2)\subset\cdots\subset \Sym(n)\subset\cdots,
$$
whose limit is the infinite subgroup
\[
  \Sym(\infty)=\bigcup_{n\geq 1} \Sym(n)\subset \Sym(\NN).
\]
Notice that $\Sym(\infty)$ consists of all finite permutations of $\NN$, i.e. permutations that fix all but finitely many elements
of $\NN$.
In this paper, for simplicity, we often use $\Sym$ as an abbreviation for $\Sym(\infty)$.


\subsection{Sym action}
\label{sec:Sym-action}

Let $\RR^\NN$ be the usual Cartesian power of $\RR$ with index set $\NN$. For any $n\in\NN$ the vector space $\RR^n$ is considered as a subspace of $\RR^\NN$, in which each $(u_1,\dots,u_n)\in\RR^n$ is identified with
\[
 (u_1,\dots,u_n,0,0,\dots)\in\RR^\NN.
\]
In this manner, we have an increasing chain of finite dimensional subspaces
\[
\RR\subset \RR^2\subset\cdots\subset\RR^n\subset\cdots,
\]
whose limit is the infinite dimensional subspace
\[
 \RR^{\infty}=\bigcup_{n\geq 1} \RR^{n}\subset\RR^\NN.
\]
Note that $\RR^{\infty}$ consists of all elements of $\RR^\NN$ of finite support. Here, the \emph{support} of an element $\ub=(u_1,u_2,\dots)\in\RR^\NN$ is defined as
\[
 \supp(\ub)=\{i\in\NN\colon u_i\ne 0\}\subseteq \NN
\]
and $|\supp(\ub)|$ is called the \emph{support size} of $\ub.$ Let
\[
 \width(\ub)=\sup\{i\in\NN\colon i\in \supp(\ub)\}
\]
denote \emph{width} of $\ub$. Then $\RR^{\infty}$ can also be described as the subset of $\RR^\NN$ consisting of elements with finite width.  Moreover,
$\RR^n=\{\ub\in\RR^\NN\colon\width(\ub)\le n\}$ for all $n\in\NN$. By definition, it is also clear that $|\supp(\ub)|\le\width(\ub)$ for every $\ub\in\RR^\NN$.

As a vector space, $\RR^{\infty}$ has a canonical basis consisting of the vectors
$\eb_{i}$, $i\in\NN$, with
\[
  (\eb_{i})_{j} =
  \begin{cases}
    1 & \text {if $i = j$},\\
    0 & \text{otherwise}.
  \end{cases}
\]
Note also that $\eb_{1},\dots,\eb_{n}$ form the canonical basis of $\RR^n$ for all $n\in\NN$.

There is a natural action of $\Sym(\NN)$ on $\RR^\NN$ given by
\[
 \sigma(\ub)=(u_{\sigma^{-1}(1)},u_{\sigma^{-1}(2)},\dots)
 \quad\text{for  }\ \sigma\in \Sym(\NN)\ \text{ and }\ \ub\in\RR^\NN.
\]
This action restricts to an action of $\Sym(\infty)$ on $\RR^\infty$, and for each $n\in\NN$, of $\Sym(n)$ on $\RR^n$.
More precisely, for any $\vb=\sum_{i\ge1}v_i\eb_i\in\RR^\infty$ (with only finitely many nonzero $v_i$) one has
\[
 \sigma(\vb)=\sum_{i\ge1}v_{\sigma^{-1}(i)}\eb_i
 =\sum_{i\ge1}v_{i}\eb_{\sigma(i)}
 \quad\text{for every }\ \sigma\in \Sym(\infty).
\]
Thus, in particular, $\sigma(\eb_j)=\eb_{\sigma(j)}$ and $|\supp(\vb)|=|\supp(\sigma(\vb))|$ for any $\vb\in\RR^\infty.$

A subset $A\subseteq\RR^\NN$ is \emph{$\Sym(\NN)$-invariant} if $A$ is stable under the action of $\Sym(\NN)$, i.e.
\[
 \Sym(\NN)(A)\defas\{\sigma(\ub)\colon \sigma\in \Sym(\NN),\ \ub\in A\}
 \subseteq A.
\]
In the same manner, one defines the notions of \emph{$\Sym$-invariant} (recall that $\Sym=\Sym(\infty)$) and \emph{$\Sym(n)$-invariant} sets for each $n\in\NN.$ Observe that if $\ub\in\RR^\infty$ and $\sigma\in\Sym(\NN)$, then there exists $\pi\in \Sym$ such that $\sigma(\ub)=\pi(\ub)$. Hence
\[
 \Sym(\NN)(A)=\Sym(A)
 \quad\text{for all }\ A\subseteq\RR^\infty.
\]
So a subset of $\RR^\infty$ is $\Sym(\NN)$-invariant precisely when it is $\Sym$-invariant.

\subsection{Chains of cones}

For a subset $A\subseteq\RR^\NN$ let $\cncl(A)$ denote the cone generated by $A$, i.e.
\[
 \cncl(A)=
 \Big\{\sum_{i=1}^k\lambda_i\ab_i\colon k\in\NN,\ \ab_i\in A,\
\lambda_i\in\RR_{\geq 0}\Big\}\subseteq\RR^\NN.
\]
Evidently, if $A\subseteq\RR^\infty$ (respectively, $A\subseteq\RR^n$), then $\cncl(A)\subseteq\RR^\infty$ (respectively, $\cncl(A)\subseteq\RR^n$). A cone is \emph{finitely generated} if it has a finite generating set.

We are interested in cones $C\subseteq\RR^\infty$ that are $\Sym$-invariant and \emph{$\Sym$-equivariantly finitely generated}, i.e.
\[
 C=\cncl(\Sym(A))
\]
for some finite set $A\subset\RR^\infty$. To understand such a cone (viewed as a \emph{global} object) one may consider a chain of cones $\C=(C_n)_{n\ge1}$ with $C_n\subseteq \RR^n$ for all $n\ge 1$ (viewed as \emph{local} objects) such that $C$ is the \emph{limit} of this chain, i.e. $C=\bigcup_{n\ge1}C_n$.
It is shown in \cite{KLR} that certain properties of $C$ and the chain $\C$ are intimately related. To state the result let us recall some further notions.

Let $\C=(C_n)_{n\ge1}$ be an arbitrary chain of cones with $C_n\subseteq\RR^n$ for all $n\ge1.$ Then $\C$ is called \emph{$\Sym$-invariant} if
\[
 \cncl(\Sym(n)(C_m))\subseteq C_n
 \quad\text{for all }\ n\ge m\ge 1.
\]
In other words, a chain of cones $\C$ is $\Sym$-invariant precisely when the following conditions are satisfied:
\begin{enumerate}
 \item
 $\C$ is an increasing chain, i.e. $C_m\subseteq C_n$ for all $1\le m\le n$.
 \item
 $C_n$ is $\Sym(n)$-invariant for all $n\ge1.$
\end{enumerate}
We say that a $\Sym$-invariant chain $\C$ \emph{stabilizes} if there exists $r\in\NN$ such that
\[
 C_n=\cncl(\Sym(n)(C_m))
 \quad\text{for all }\ n\ge m\ge r.
\]
The smallest such $r$ is called the \emph{stability index} of $\C$ and denoted by $\ind(\C)$.

It is evident that a global cone $C\subseteq\RR^\infty$ is $\Sym$-invariant if an associated chain of local cones $\C=(C_n)_{n\ge1}$
is $\Sym$-invariant in the above sense. Moreover, when the local cones are nonnegative, their stabilization and finite generation characterize the equivariant finite generation of the global cone, as exhibited in the next result. This is a special instance of local-global principles studied in \cite{KLR}.

\begin{lemma}[{\cite[Corollary 5.4]{KLR}}]
	\label{l:globallocalcones}
	Let $\C=(C_{n})_{n\geq 1}$ be a $\Sym$-invariant chain of cones with
	$C_{n} \subseteq \RR^{n}_{\geq 0}$ for all $n\ge1$. Let $C$ denote the limit cone.  Then
	the following are equivalent:
	\begin{enumerate}
		\item
		$C$ is $\Sym$-equivariantly finitely generated;
		\item
		$\C$ stabilizes and $C_{n}$ is finitely generated for all $n\gg0$;
		\item
		There exists an $s\in\NN$ such that for all $n\ge s$ the following hold:
		\begin{enumerate}
			\item
			$C \cap \RR^{n}= C_{n}$,
			\item
			$C_{n}$ is finitely generated by elements of support size at most $s$.
		\end{enumerate}
	\end{enumerate}
\end{lemma}

Given a global cone $C\subseteq\RR^\infty$, the largest chain of local cones with limit $C$ is the \emph{saturated} chain $\overline{\C}=(\overline{C}_{n})_{n\geq 1}$ defined by $\overline{C}_n=C\cap \RR^n$ for all $n\ge1$.
Note, however, that when $C\subseteq\RR^\infty_{\ge0}$ is $\Sym$-equivariantly finitely generated, any $\Sym$-invariant chain of cones $\C=(C_{n})_{n\geq 1}$ with limit $C$ does not differ too much from the saturated chain $\overline{\C}$. In fact, it follows from Lemma~\ref{l:globallocalcones} that $\C$ is \emph{eventually saturated}, i.e. $C_n=C\cap \RR^n$ for all $n\gg0$. Thus, $\C$ coincides with $\overline{\C}$ up to finitely many members.

In this paper, Lemma~\ref{l:globallocalcones} is mostly applied to the saturated chain $\overline{\C}$. Since the result in that case has a rather more concise form, we state it here for the sake of convenience.

\begin{corollary}
  \label{c:globallocalcones}
  Let $C\subseteq\RR^\infty_{\ge0}$ be a $\Sym$-invariant cone and let $\overline{\C}=(\overline{C}_{n})_{n\geq 1}$ be its saturated chain of local cones.  Then the following are equivalent:
  \begin{enumerate}
  \item
    $C$ is $\Sym$-equivariantly finitely generated;
  \item
    $\overline{\C}$ stabilizes and $\overline{C}_{n}$ is finitely generated for all $n\ge1$;
  \item
    There exists an $s\in\NN$ such that $\overline{C}_{n}$ is finitely generated by elements of support size at most $s$ for all $n\ge 1$.
  \end{enumerate}
\end{corollary}

\begin{proof}
 In view of Lemma \ref{l:globallocalcones} it remains to prove the implication (i)$\Rightarrow$(iii). To this end, suppose
 $C=\cncl(\Sym(A))$
 with $A=\{\ab_1,\dots,\ab_k\}\subset\RR^\infty_{\ge0}.$ By symmetry, we may assume that $\ab_i\in\RR^{s_i}$ with $s_i=|\supp(\ab_i)|$ for $i=1,\dots,k$. Setting $A_n=A\cap\RR^n$ it is enough to show that
 $\overline{C}_n=\cncl(\Sym(n)(A_n))$ for $n\ge1$. The inclusion ``$\supseteq$'' is clear since $\overline{C}_n=C\cap \RR^n$. 
 
 For the reverse inclusion, let $\ub\in \overline{C}_n.$ Then there exist $\lambda_1,\dots,\lambda_m>0$ and $\sigma_1,\dots,\sigma_m\in\Sym$ such that $\ub=\sum_{j=1}^m\lambda_j\sigma_j(\ab_{i_j})$ for some $i_1,\dots,i_m\in[k].$ Since each $\sigma_j(\ab_{i_j})$ belongs to $\RR^{\infty}_{\geq 0}$, we must have
 $$
  \width(\sigma_j(\ab_{i_j}))\le\width(\ub)\le n,
 $$
 i.e. $\sigma_j(\ab_{i_j})\in \RR^n$ for $j=1,\dots,m.$ Moreover, from
$$
 s_{i_j}=|\supp(\ab_{i_j})|=|\supp(\sigma_j(\ab_{i_j}))|
 \le \width(\sigma_j(\ab_{i_j}))\le n
$$
it follows that
$
\ab_{i_j}\in A\cap\RR^{s_{i_j}}\subseteq A\cap\RR^n=A_n
$
for $j=1,\dots,m.$
Now that  $\ab_{i_j}$ and $\sigma_j(\ab_{i_j})$ both belong to $\RR^n$, we can find $\pi_j\in\Sym(n)$ such that $\sigma_j(\ab_{i_j})=\pi_j(\ab_{i_j})$. Hence,
\[
 \ub=\sum_{j=1}^m\lambda_j\sigma_j(\ab_{i_j})
 =\sum_{j=1}^m\lambda_j\pi_j(\ab_{i_j})
 \in\cncl(\Sym(n)(A_n))
\]
as desired.
\end{proof}


\section{Equivariant Carath\'{e}odory's theorem}
\label{sec-Caratheodory}

A classical theorem of Carath\'{e}odory states that any element of a finite dimensional cone $C=\cncl(A)\subseteq\RR^n$ is a conical combination of at most $\dim C$ elements in $A$ (see, e.g., \cite[Theorem 1.55]{BG}). The main goal of this section is to prove the following equivariant version of this result.

\begin{theorem}
\label{thm-Caratheodory}
 Let $C\subseteq \RR^{\infty}$ be a $\Sym$-invariant cone that is $\Sym$-equivariantly generated by a set $A\subseteq\RR^{\infty}_{\geq 0}$. Then any element $\ub\in C$ with support size $s=|\supp(\ub)|$ can be written as a conical combination of at most $s$ elements in $\Sym(A)$.
\end{theorem}

\begin{proof}
Evidently, we can find a permutation $\sigma\in \Sym$ such that $\sigma(\ub)\in\RR^s.$
By assumption, $\sigma(\ub)\in C=\cncl(\Sym(A))$. So there exist $\ab_1,\dots,\ab_k\in\Sym(A)$ and $\lambda_1,\dots,\lambda_k>0$ such that
\[
 \sigma(\ub)=\sum_{i=1}^k\lambda_i\ab_i.
\]
Since $A\subseteq\RR^{\infty}_{\geq 0}$, this representation implies that
\[
 \width(\ab_i)\le\width(\ub)\le s\
 \text{ for }\ i=1,\dots,k.
\]
Hence, $\ab_i\in\RR^s$ for $i=1,\dots,k.$ Note that $\sigma(\ub)\in\cncl(\ab_1,\dots,\ab_k)$. So by Cara\-th\'{e}o\-dory's theorem, there exists a subset $B\subseteq\{\ab_1,\dots,\ab_k\}$ with $|B|\le s$ such that $\sigma(\ub)\in\cncl(B).$ It follows that
\[
 \ub=\sigma^{-1}(\sigma(\ub))\in\cncl(\sigma^{-1}(B))
\]
with $\sigma^{-1}(B)=\{\sigma^{-1}(\bb)\colon \bb\in B\}\subseteq\Sym(A)$. This concludes the proof.
\end{proof}

In view of Carath\'{e}odory's theorem, one might expect that Theorem~\ref{thm-Caratheodory} continues to hold without the assumption $A\subseteq\RR^{\infty}_{\geq 0}$, or the bound by the support size $s=|\supp(\ub)|$ (which depends on $\ub$) can be replaced by a uniform bound independent of $\ub$. Unfortunately, this is not true, as illustrated by the following examples.

\begin{example}
	Consider the cone $C=\cncl(\Sym(\ub_1,\ub_2))$ with
	\[
	\ub_1=(1,1),\quad \ub_2=(1,-1).
	\]
	Then every element of $\Sym(\ub_1,\ub_2)$ has support size 2. Thus, $\eb_1=\frac{1}{2}\ub_1+\frac{1}{2}\ub_2\in C$, having support size 1, cannot be a conical combination of one element of $\Sym(\ub_1,\ub_2)$.
\end{example}

\begin{example}
\label{ex:no-bound}
	Let $C=\RR^{\infty}_{\geq 0}=\cncl(\Sym(\eb_1))$. Then for any $n\ge1$ the element
	\[
	\ub_n=\eb_1+\cdots+\eb_n
	\]
	cannot be a conical combination of less than $n$ elements of $\Sym(\eb_1)=\{\eb_i\colon i\in\NN\}.$
\end{example}

\section{Equivariant Minkowski-Weyl theorem}
\label{sec-Minkowski-Weyl}

The Minkowski-Weyl theorem says that any finitely generated cone $C\subseteq\RR^n$ can also be represented as an intersection of finitely many linear halfspaces and vice versa (see, e.g., \cite[Theorem 1.15]{BG}). In other words, the dual cone of any finitely generated cone in $\RR^n$ is again finitely generated. As we will see in Theorem~\ref{thm:global-dual}, this result cannot be extended ``globally'' to every $\Sym$-equivariantly finitely generated cone $C\subseteq\RR^\infty$. Instead, we obtain Theorem~\ref{t:Minkowski-Weyl} that can be seen as a \emph{local} extension of the Minkowski-Weyl theorem to the equivariant setting.

In what follows we identify the dual vector space of $\RR^{\infty}$ with $\RR^\NN$ via the dual pairing
\begin{align*}
\langle\cdot,\cdot\rangle:\RR^{\infty}\times\RR^\NN&\to\RR,\\
(\ub,\vb)&\mapsto \langle\ub,\vb\rangle=\sum_{i\ge1}u_iv_i,
\end{align*}
where $\ub=(u_i)\in\RR^{\infty}$ and $\vb=(v_i)\in\RR^{\NN}$. Restricting this dual pairing to $\RR^n\times\RR^n$ gives the usual identification of the dual space of $\RR^n$ with itself for every $n\in\NN$.

\begin{definition}
	Let $C\subseteq\RR^{\infty}$ be a cone. The \emph{dual cone} of $C$ is defined as
	\[
	C^*=\{\ub\in \RR^{\NN}\colon \langle \ub,\vb\rangle\ge 0\
	\text{ for all }\ \vb\in C\}.
	\]
	When $C\subseteq\RR^n$ for a fixed $n\in\NN$, we define $C^*\subseteq\RR^n$ to be the standard dual cone of $C$, i.e.
	\[
	C^*=\{\ub\in \RR^{n}\colon \langle \ub,\vb\rangle\ge 0\
	\text{ for all }\ \vb\in C\}.
	\]
\end{definition}

\begin{remark}
	If $C$ is a cone in $\RR^{n}$, then $C$ is also a cone in $\RR^{m}$ for any $m\ge n$. So the above definition of the dual cone of $C$ \emph{depends} on the ambient space in which $C$ lives. Nevertheless, this will cause no confusion, because for a given chain of cones  $\C=(C_{n})_{n\geq 1}$ we always mean the embedding $C_n\subseteq\RR^n$ when considering the dual cone of $C_n$.
\end{remark}

\begin{lemma}
	\label{l:dual-cone}
	The following statements hold for a cone $C$:
	\begin{enumerate}
		\item
		If $C\subseteq\RR^{\infty}$ is $\Sym$-invariant, then $C^*\subseteq\RR^{\NN}$ is also $\Sym$-invariant.
		\item
		If $C\subseteq\RR^n$ is $\Sym(n)$-invariant, then $C^*\subseteq\RR^n$ is also $\Sym(n)$-invariant.
	\end{enumerate}
\end{lemma}

\begin{proof}
We only prove (i), since the proof of (ii) is similar.	Let $\ub\in C^*$ and $\sigma\in \Sym$. For any $\vb\in C$, one has $\sigma^{-1}(\vb)\in C$ because $C$ is $\Sym$-invariant. Hence,
	\[
	\langle \sigma(\ub),\vb\rangle
	=\sum_{i\ge1}u_{\sigma^{-1}(i)}v_i
	=\sum_{i\ge1}u_iv_{\sigma(i)}
	=\langle \ub,\sigma^{-1}(\vb)\rangle\ge0.
	\]
Since this is true for every $\vb\in C$, we deduce that $\sigma(\ub)\in C^*$. Thus, $C^*$ is $\Sym$-invariant.
\end{proof}

In order to study the Minkowski-Weyl theorem in the equivariant setting, one needs to understand the dual cone $C^*$ of a $\Sym$-equivariantly finitely generated cone $C\subseteq\RR^\infty$. This turns out to be difficult because $C^*$ is not $\Sym$-equivariantly finitely generated (see Theorem~\ref{thm:global-dual}). On the other hand, it is possible to describe the chain $\C^*=(C_{n}^*)_{n\geq 1}$ of dual cones of a chain $\C=(C_{n})_{n\geq 1}$ of local cones associated to $C$, provided that $C$ is nonnegative. In the trivial case $C=\{\nub\}$, one has $C_{n}^*=\RR^n$ for all $n\geq 1$. For ease of presentation, we will exclude this case and assume that the chain $\C=(C_{n})_{n\geq 1}$ is \emph{non-trivial}, that is, $C_{n}\ne \{\nub\}$ for some $n\ge1$. 
In what follows, by abuse of notation, for $\ub=(u_1,\dots,u_r)\in\RR^r$ and $u_{r+1},\dots,u_n\in\RR$ we write $(\ub,u_{r+1},\dots,u_n)$ to mean the element $(u_1,\dots,u_r,u_{r+1},\dots,u_n)\in\RR^n.$

\begin{theorem}
	\label{t:Minkowski-Weyl}
	Let $\C=(C_{n})_{n\geq 1}$ be a stabilizing non-trivial $\Sym$-invariant chain of finitely generated cones $C_{n} \subseteq \RR^{n}_{\geq 0}$ with $r=\ind(\C).$ Let $\C^*=(C_{n}^*)_{n\geq 1}$ be the chain of dual cones with $C_{n}^* \subseteq \RR^{n}$ for $n\ge1$. Then a generating set of $C_r^*$ completely determines that of $C_n^*$ for all $n\ge r.$ More precisely, there exists a finite subset $F_r\subset C_r^*$ with the following properties:
	\begin{enumerate}
	 \item
	 If $\ub=(u_1,\dots,u_r)\in F_r$, then $u_1\le\cdots\le u_r$.
	 \item
	 For all $n\ge r$ it holds that
	 \begin{equation}
		\label{e:dual-cone}
		C_n^*=\cncl(\Sym(n)(F_n\cup\{\eb_n\})),
	\end{equation}
	where
	\begin{equation}
		\label{e:generating-set}
		F_{n}
		=\{(\ub,u_{r},\dots,u_r)\in \RR^{n}\colon \ub=(u_1,\dots,u_r)\in F_r\}
		\quad \text{for }\ n\ge r+1.
	\end{equation}
	\end{enumerate}
\end{theorem}

The proof of this result requires some preparation.

\begin{lemma}
	\label{l:dual-chain}
	Let $\C=(C_{n})_{n\geq 1}$ be a stabilizing non-trivial $\Sym$-invariant chain of cones with $C_{n} \subseteq \RR^{n}_{\geq 0}$ for all $n\ge 1$. Suppose $r=\ind(\C)$ and let $G_r$ be the subset of $C_r^*$ consisting of all elements $\ub=(u_1,\dots,u_r)\in\RR^{r}$ with $u_1\le\cdots\le u_r$ and $|\supp(\ub)|\ge 2.$
	Set
	\[
	G_{n}
		=\{(\ub,u_{r},\dots,u_r)\in \RR^{n}\colon \ub=(u_1,\dots,u_r)\in G_r\}
	\quad \text{for }\ n\ge r+1.
	\]
	Then
	\begin{equation}
	\label{e:dual-cone-2}
	 C_n^*=\cncl(\Sym(n)(G_n\cup\{\eb_n\}))
	\quad \text{for all }\ n\ge r.
	\end{equation}
\end{lemma}

\begin{proof}
	We first prove the equality for $n=r$. Note that $C_r\ne \{\nub\}$ because $\C$ is non-trivial. Since $C_r\subseteq\RR_{\ge0}^r$, it is evident that $\eb_r\in C_r^*$. So Lemma~\ref{l:dual-cone} yields the inclusion ``$\supseteq$''.
	For the reverse one, let $\vb\in C_r^*$ with $\vb\ne \nub$. If $|\supp(\vb)|=1$, then $\vb=\lambda\sigma(\eb_r)$ for some $\lambda>0$ and $\sigma\in\Sym(r)$. Otherwise, $|\supp(\vb)|\ge 2$ and thus $\vb=\pi(\ub)$ for some $\ub\in G_r$ and $\pi\in\Sym(r)$. This establishes the reverse inclusion.

Next we consider the case $n\ge r+1$. For the inclusion ``$\supseteq$'', using induction on $n$ we may assume that $n=r+1.$ By Lemma~\ref{l:dual-cone}, it suffices to show that $G_{r+1}\subseteq C_{r+1}^*.$ 

Let $\vb=(\ub,u_r)\in G_{r+1}$ for some $\ub=(u_1,\dots,u_r)\in G_r.$ Since
	$C_{r+1}=\cncl(\Sym(r+1)(C_r))$, we only need to check that
	\[
	\langle \vb,\pi(\wb)\rangle\ge0
	\quad\text{for every }\ \wb\in C_r\ \text{ and }\ \pi\in\Sym(r+1).
	\]
	Write $\pi(\wb)=(x_1,\dots,x_{r+1})$. By definition of the action of $\pi$, there exists $i\in[r+1]$ with $x_i=0$. Moreover, we can choose
	$\sigma\in\Sym(r)$ such that
	\[
	(x_1,\dots,x_{i-1},x_{i+1},\dots,x_{r+1})=\sigma(\wb)\in C_r.
	\]
	Now let
	\[
	\ub'=(u_1,\dots,u_{i-1},u_{i+1},\dots,u_r, u_r)
	=\ub+\sum_{j=i+1}^{r}(u_{j}-u_{j-1})\eb_{j-1}.
	\]
	Then
	$
	\ub'\in C_r^*,
	$
	since $\ub\in G_r\subseteq C_r^*$ and $\eb_{j}\in  C_r^*$ for all $j\in[r].$
	It follows that
	\[
	\langle \vb,\pi(\wb)\rangle=\sum_{j=1}^ru_ix_{i}+u_rx_{r+1}
	= \langle \ub',\sigma(\wb)\rangle\ge0.
	\]
	
	It remains to prove the inclusion ``$\subseteq$'' for all $n\ge r+1.$ Let $\vb=(v_1,\dots,v_n)\in C_n^*$. Since both sides of \eqref{e:dual-cone-2} are  $\Sym(n)$-invariant, we may assume that
	$v_1\le\cdots\le v_n.$ Let
	$
	\vb'=(v_{1},\dots,v_r).
	$
	We show that $\vb'\in  C_r^*.$ Indeed, recall that each $\wb'\in C_r$ is identified with
	\[
	\wb=(\wb',0,\dots,0)\in C_n.
	\]
	Hence,
	\[
	\langle \vb',\wb'\rangle=\langle \vb,\wb\rangle\ge 0
	\]
	and so $\vb'\in C_r^*.$ If $|\supp(\vb')|\le 1$, then we must have either
	$
	0= v_{1}=\cdots=v_{r-1}\le v_r
	$
	or
	$
	0=v_{r}=\cdots=v_{2}>v_{1}.
	$
	The latter cannot happen since $\langle \vb',\wb'\rangle\ge0$ for all $\wb'\in C_r\subseteq\RR^r_{\ge0}$.
	Thus, the former holds. It follows that $\vb\in \RR^{n}_{\geq 0}$ and so
	\[
	\vb\in \cncl(\Sym(n)(\eb_n))
	\subseteq \cncl(\Sym(n)(G_n\cup\{\eb_n\})).
	\]
	If $|\supp(\vb')|\ge 2$, then $\vb'\in G_r$ by definition of $G_r.$ Hence, for $\ub=(\vb',v_{r},\dots,v_{r})\in \RR^{n}$ we get $u\in G_n$. This implies
	\[
	\vb=\ub+\sum_{j=r+1}^{n}(v_{j}-v_{r})\eb_j
	\in\cncl(\Sym(n)(G_n\cup\{\eb_n\})),
	\]
	concluding the proof of the inclusion ``$\subseteq$''.
\end{proof}

\begin{lemma}
	\label{l:decreasing-subcone}
	If $C_n\subseteq \RR^n$ is a finitely generated cone, then its subcone of elements with non-decreasing coordinates
	\[
	D_n=\{(u_1,\dots,u_n)\in C_n\colon u_1\le\cdots\le u_n\}
	\]
	is also finitely generated.
\end{lemma}

\begin{proof}
	$D_n$ is the intersection of $C_n$ with the following linear halfspaces
	\[
	H_i^{+}=\{(x_1,\dots,x_n)\in\RR^n\colon x_i\le x_{i+1}\},
	\ \ i=1,\dots, n-1.
	\]
	Therefore, $D_n$ is finitely generated by the classical Minkowski-Weyl theorem.
\end{proof}

Now we are ready to prove Theorem~\ref{t:Minkowski-Weyl}.

\begin{proof}[Proof of Theorem~\ref{t:Minkowski-Weyl}]
	For $n\ge r$ let $G_n$ be the subset of $C_n^*$ as in Lemma~\ref{l:dual-chain}. Let $D_r^*$ be the subcone of $C_r^*$ consisting of elements with non-decreasing coordinates. Notice that $G_r$ is a subset of $D_r^*$ that consists of elements with support size at least 2. Since $\eb_r\in D_r^*$, it is easy to see that
	\[
	D_r^*=\cncl(G_r\cup\{\eb_r\})=G_r\cup\cncl(\eb_r).
	\]
	By Lemma~\ref{l:decreasing-subcone}, $D_r^*$ is finitely generated. Hence, there exists a finite subset $F_r\subset G_r$ such that
	\begin{equation}
		\label{e:subcone}
		D_r^*=\cncl(F_r\cup\{\eb_r\}).
	\end{equation}
	For $n\ge r+1$ let $F_n$ be defined as in \eqref{e:generating-set}. Of course every element of $F_r$ has non-decreasing coordinates. So it remains to prove \eqref{e:dual-cone}. Using Lemma~\ref{l:dual-chain} it suffices to verify that
	\[
	G_n\subseteq \cncl(F_n\cup\{\eb_n\})
	\quad \text{for all }\ n\ge r.
	\]
	By \eqref{e:subcone}, this is true for $n=r$. Now suppose that $n\ge r+1$ and let $\vb=(\ub,u_{r},\dots,u_r)\in G_n$ with $\ub=(u_1,\dots,u_r)\in G_r.$ Since $G_r\subseteq \cncl(F_r\cup\{\eb_r\})$, we can find $\wb_1,\dots,\wb_m\in F_r$ and $\lambda,\lambda_1,\dots,\lambda_m\ge0$ such that
	\[
	\ub=\sum_{i=1}^m\lambda_i\wb_i+\lambda \eb_r.
	\]
	Write $\wb_i=(w_{i1},\dots,w_{ir})$ and set $\vb_i=(\wb_i,w_{ir},\dots,w_{ir})\in \RR^n$ for $i=1,\dots, m$. Then $\vb_i\in F_n$ and one has
	\[
	\vb=\sum_{i=1}^m\lambda_i\vb_i+\lambda\sum_{j=r}^{n}\eb_j
	\in\cncl(F_n\cup\{\eb_n\}).
	\]
	Hence, $G_n\subseteq \cncl(F_n\cup\{\eb_n\})$, as desired.
\end{proof}

We illustrate Theorem~\ref{t:Minkowski-Weyl} by the following simple example that was also obtained in \cite[Proposition 2.2.4]{An} by a more direct approach.

\begin{example}
	\label{e:1-generator}
	Let $\ub=(1,a)\in\RR^2$ with $a\ge 1$. Consider the $\Sym$-invariant chain of cones $\C=(C_{n})_{n\geq 1}$ with $C_1=0$ and $C_n=\cncl(\Sym(n)(\ub))$ for all $n\ge2.$
	It is evident that $C_2^*=\cncl(\Sym(2)((-1,a)))$, and the subcone $D_2^*\subseteq C_2^*$ of elements with non-decreasing coordinates is generated by $F_2=\{(-1,a),(1,1)\}$. So by Theorem~\ref{t:Minkowski-Weyl},
	\[
	C_n^*=\cncl(\Sym(n)(F_n\cup\{\eb_n\}))
	\ \text{ with }\
	F_n=\{(-1,a,\dots,a),(1,\dots,1)\}
	\ \text{ for }\ n\ge 3.
	\]
	Since $(1,\dots,1)\in \cncl(\Sym(n)(\eb_n))$, this generator is superfluous and one gets the following simpler representation
	\[
	C_n^*=\cncl(\Sym(n)(F_n'\cup\{\eb_n\}))
	\ \text{ with }\
	F_n'=\{(-1,a,\dots,a)\}
	\ \text{ for }\ n\ge 3.
	\]
\end{example}

It would be of interest to extend Theorem~\ref{t:Minkowski-Weyl} to $\Sym$-invariant chains of cones without the nonnegativity assumption.

\begin{problem}
 Let $\C=(C_{n})_{n\geq 1}$ be a stabilizing $\Sym$-invariant chain of finitely generated cones $C_{n} \subseteq \RR^{n}$. Describe the chain $\C^*=(C_{n}^*)_{n\geq 1}$ of dual cones.
\end{problem}

In the remainder of this section we discuss the dual of a $\Sym$-invariant cone $C\subseteq\RR_{\ge0}^{\infty}$. Consider the saturated chain of local cones $\overline{\C}=(\overline{C}_{n})_{n\geq 1}$ of $C$. When $C$ is $\Sym$-equivariantly finitely generated, Theorem~\ref{t:Minkowski-Weyl} provides a satisfying description of the chain $\overline{\C}^*=(\overline{C}_{n}^*)_{n\geq 1}$ of dual cones. With notation of this result one might expect that the global dual cone $C^*\subseteq\RR^\NN$ is $\Sym$-equivariantly generated by the finite set $F\cup\{\eb_1\}$, where
\[
 F=\{(\ub,u_{r},u_r,\dots)\in \RR^{\NN}\colon \ub=(u_1,\dots,u_r)\in F_r\}.
\]
However, somewhat surprisingly, this is far from the truth. In fact, the proof of the next result shows that $C^*$ is much bigger than any subcone that is $\Sym$-equivariantly finitely generated. Notice that this result does not require the assumption that $C$ is $\Sym$-equivariantly finitely generated.

\begin{theorem}
\label{thm:global-dual}
	Let $C\subseteq\RR_{\ge0}^{\infty}$ be a $\Sym$-invariant cone. Then the dual cone $C^*\subseteq\RR^\NN$ is not
	$\Sym$-equivariantly finitely generated.
\end{theorem}

\begin{proof}
    Let $D$ be an arbitrary subcone of $C^*$ that is $\Sym$-equivariantly  generated by a finite set $A$. Let $V=\Span(D)$ be the linear subspace of $\RR^\NN$ generated by $D$. Then $V=\Span(\Sym(A))$. Since $\Sym$ is countable and $A$ is finite, $\dim V$ is at most countable. On the other hand, one has $\RR_{\ge0}^\NN\subseteq C^*$ since $C\subseteq\RR_{\ge0}^{\infty}$. Thus, $\Span(C^*)=\RR^\NN$ has uncountable dimension. It follows that $\Span(D)\subsetneq \Span(C^*)$, and hence $D\subsetneq C^*.$
\end{proof}

\begin{remark}
 The proof of Theorem~\ref{thm:global-dual} is still valid if $A$ is countable. Therefore, under the assumption of this result, the dual cone $C^*$ is even not $\Sym$-equivariantly \emph{countably} generated.
\end{remark}

To conclude this section, we provide a candidate for the global object of a chain of dual cones.

\begin{remark}
\label{r:dual-cone}
 We keep the assumption of Theorem~\ref{t:Minkowski-Weyl}. Let $C\subseteq\RR_{\ge0}^{\infty}$ denote the limit cone of $\C$ and $C^*\subseteq\RR^\NN$ its dual cone. As we have seen in Theorem~\ref{thm:global-dual}, $C^*$ is too big and does not reflect well the properties of the chain $\C^*=(C_{n}^*)_{n\geq 1}$ of dual cones. So one might ask, what could be the ``global'' cone of $\C^*$? The first natural candidate is of course the union $\bigcup_{n\ge1}C_{n}^*$. But this cone is not $\Sym$-invariant, and therefore, not really interesting. Another natural candidate is the inverse limit
 \[
  \Cf^*=\varprojlim_{n\ge1} C_{n}^*.
 \]
Here, $\C^*$ is an inverse system with the maps $\varphi_n:C_n^*\to C_{n-1}^*$, in which $\varphi(\ub)\in C_{n-1}^*$ is obtained from $\ub\in C_n^*$ by removing one maximal coordinate of $\ub$. The cone $\Cf^*$ can be identified with a subcone of $C^*$ that satisfies the following properties:
\begin{enumerate}
 \item
 each $\ub\in\Cf^*$, viewed as a sequence of its coordinates, is a well-ordered sequence (i.e. every non-empty subsequence of $\ub$ has a least element), and
 \item
 if $\ub'$ is the subsequence of $\ub$ consisting of $r$ smallest elements of $\ub$, then $\ub'\in C_r^*$ (recall that $r=\ind(\C)$).
\end{enumerate}
Evidently, $\Cf^*$ is $\Sym$-invariant. Moreover, $\Cf^*$ is a proper subcone of $C^*$ because
$\ub=(1/n)_{n\ge1}\in C^*\setminus \Cf^*$. It would be interesting to know whether $\Cf^*$ is $\Sym$-equivariantly finitely generated and, more generally, how the properties of the chain $\C^*$ are encoded in $\Cf^*.$
\end{remark}

\section{Chains of monoids and normality}
\label{sec:monoid}

As a preparation for the next section, we discuss here chains of (normal) monoids. It is shown that the results for cones in Section~\ref{section-cones}  have their monoid counterparts. In particular, the close relationship between cones and normal monoids is emphasized.

Recall that the monoid generated by $A\subseteq\ZZ^{\NN}$ is the submonoid
$$
\mndcl(A)=\Big\{\sum_{i=1}^km_i\ab_i\colon k\in\NN,\ \ab_i\in A,\
m_i\in\ZZ_{\geq 0}\Big\}\subseteq\ZZ^{\NN}.
$$
Replacing the conical operation $\cncl$ with the monoidal operation $\mndcl$, the notions defined for cones in Section~\ref{section-cones}, such as \emph{$\Sym$-equivariant finite generation, $\Sym$-invariant chains, stabilizing chains}, etc., extend without difficulty to monoids.
A monoid $M\subseteq\ZZ^{\NN}$ is \emph{positive} if the only element of $M$ with the property that $\ub\in M$ and $-\ub\in M$ is $\nub.$ This is the case, e.g., when $M\subseteq\ZZ^{\NN}_{\ge0}$. It is known that if $M\subseteq\ZZ^{n}$ (for some $n\in\NN$) is positive and finitely generated, then $M$ has a unique minimal system of generators, called the \emph{Hilbert basis} of $M$ and denoted by $\Hc(M)$ (see, e.g., \cite[Definition 2.15]{BG}). Moreover, in this case, $\Hc(M)$ consists of exactly the \emph{irreducible} elements of $M$, i.e. those elements $\ub\in M\setminus\{\nub\}$ with the property that if $\ub=\vb+\wb$ for $\vb,\wb\in M$, then either $\vb=\nub$ or $\wb=\nub$.

As a monoid counterpart of Corollary~\ref{c:globallocalcones}, the following result in fact says more: it also characterizes the stabilization of a chain of monoids through the stabilization of Hilbert bases. This new characterization, though simple, is essential for the next section. Let $\|\cdot\|$ denote the $1$-norm on $\RR^n$, i.e. $\|\ub\|=|u_1|+\cdots+|u_n|$ for any $\ub=(u_1,\dots,u_n)\in\RR^n.$

\begin{lemma}
  \label{l:globallocalmonoids}
  Let $M\subseteq\ZZ^\infty_{\ge0}$ be a $\Sym$-invariant monoid and let
  $\overline{\M}=(\overline{M}_{n})_{n\geq 1}$ with
  $\overline{M}_{n}=M\cap\ZZ^n$ for $n\ge1$ be its saturated chain of local monoids. Let $\Hc_n$ denote the Hilbert basis of $\overline{M}_n$ for $n\ge1$.
  Then the following are equivalent:
	\begin{enumerate}
		\item
		$M$ is $\Sym$-equivariantly finitely generated;
		\item
		$\overline{\M}$ stabilizes and $\overline{M}_{n}$ is finitely generated for all $n\ge1$;
		\item
		There exists an $s\in\NN$ such that $\overline{M}_{n}$ is finitely generated by elements of support size at most $s$ for all $n\ge1$;
		\item
        $\Hc_n$ is finite and
        $\Hc_n=\Sym(n)(\Hc_{m})$ for all $n\ge m\gg0$;
        \item
        $\|\Hc_n\|= \|\Hc_{m}\|<\infty$ for all $n\ge m\gg0$, where
        $
        \|\Hc_n\|=\sup\{\|\ub\|: \ub\in\Hc_n\};
        $
		\item
        $\|\Hc_n\|\le \|\Hc_{m}\|<\infty$ for all $n\ge m\gg0$.
	\end{enumerate}
\end{lemma}

\begin{proof}
 The equivalence of (i), (ii), and (iii) follows from \cite[Corollary 5.13]{KLR} by an argument analogous to the one used in the proof of Corollary~\ref{c:globallocalcones}.

 (ii)$\Rightarrow$(iv): Obviously, $\Hc_n$ is finite for all $n\ge1$. For the stabilization of Hilbert bases we first show that $\Sym(n)(\Hc_{m})\subseteq\Hc_n$ for all $n\ge m\ge1$. Since $\overline{M}_{n}$ is $\Sym(n)$-invariant, $\Hc_n$ is also $\Sym(n)$-invariant. So we only need to check that $\Hc_{m}\subseteq\Hc_n$. Let $\ub\in \Hc_{m}$ and suppose $\ub=\ub_1+\ub_2$ with $\ub_1,\ub_2\in \overline{M}_{n}.$ Then the nonnegativity of the $\ub_i$ yields
 \[
  \width(\ub_i)\le\width(\ub)\le m
  \quad\text{for }\ i=1,2.
 \]
Hence, $\ub_i\in\overline{M}_{n}\cap\ZZ^m=\overline{M}_{m}$ for $i=1,2$. Since $\ub\in \Hc_{m}$, we conclude that either $\ub_1=\nub$ or $\ub_2=\nub$, and therefore $\ub\in \Hc_{n}.$

Let us now prove the reverse inclusion for $n\ge m\gg0$ when $\overline{\M}$ stabilizes. Since $\Hc_m$ is a generating set of $\overline{M}_m$, it is evident that $\Sym(n)(\Hc_{m})$ is a generating set of
$\mndcl(\Sym(n)(\overline{M}_m))$. So when $\overline{M}_{n}=\mndcl(\Sym(n)(\overline{M}_m))$, we must have $\Hc_n\subseteq\Sym(n)(\Hc_{m})$.

(iv)$\Rightarrow$(v)$\Rightarrow$(vi): This is clear because $\|\ub\|=\|\sigma(\ub)\|$ for every $\ub\in\Hc_m$ and $\sigma\in\Sym(n)$.

(vi)$\Rightarrow$(iii): It is shown in the proof of the implication (ii)$\Rightarrow$(iv) that $\Sym(n)(\Hc_{m})\subseteq\Hc_n$ for all $n\ge m\ge1$. This implies $\|\Hc_m\|\le \|\Hc_{n}\|$ for all $n\ge m\ge1$. So by assumption,
\[
 s\defas\sup\{\|\Hc_n\|:n\ge1\}<\infty.
\]
From $\|\Hc_{n}\|<\infty$ it follows that $\Hc_n$ is a finite set for every $n\ge1$. Furthermore, it is apparent that $|\supp(\ub)|\le\|\ub\|\le s$ for all $\ub\in\Hc_n$ and $n\ge1$. Hence, $\overline{M}_n$ is generated by finitely many elements of support size at most $s$ for all $n\ge1$.
\end{proof}

Let us now turn to normal monoids. For a monoid $M\subseteq\ZZ^{\NN}$ the \emph{normalization} of $M$ is defined to be
\[
 \widehat{M}=\{\ub\in\ZZ^{\NN}\colon k\ub\in M\ \text{ for some }\ k\in\NN\},
\]
and $M$ is called \emph{normal} if $M=\widehat{M}.$ One of the primary motivations for studying normal monoids comes from their intimate relationship with cones. In fact, for any cone $C\subseteq\RR^{\NN}$, the intersection $M=C\cap\ZZ^{\NN}$ is a normal monoid. Conversely, if $M\subseteq\ZZ^{\NN}$ is a monoid, then $\widehat{M}=C\cap\ZZ^{\NN}$ with $C=\cncl(M)\subseteq\RR^{\NN}$, and so in particular,
\begin{equation}
\label{eq_normal}
  M=\cncl(M)\cap\ZZ^{\NN}  \
  \text{ when $M$ is normal.}
\end{equation}
Indeed, the inclusion $\widehat{M}\subseteq C\cap\ZZ^{\NN}$ is clear. For the reverse inclusion take $\ub\in C\cap\ZZ^{\NN}.$
Then $\ub=\sum_{i=1}^n\lambda_i\vb_i$ with $\lambda_i\ge0$ and $\vb_i\in M$. By (classical) Carath\'{e}odory's theorem, we can assume that the $\vb_i$ are linearly independent, and in that case, the $\lambda_i$ are uniquely determined and necessarily rational. Thus there exists $k\in\NN$ such that $k\ub=\sum_{i=1}^nm_i\vb_i$ with $m_i\in\ZZ_{\ge0}$. This means that $k\ub\in M$, and hence $\ub\in\widehat{M}.$

For $A\subseteq\ZZ^{\NN}$ the normalization of the monoid $\mndcl(A)$ is denoted by $\widehat{\mndcl}(A).$ In analogy to the cases of cones and monoids considered so far, one can also study normal monoids by using the operation $\widehat{\mndcl}$. To distinguish between $\mndcl$ and $\widehat{\mndcl}$ in analogous notions, we will add the prefix ``$\widehat{\mndcl}$'' when this operation is being used. More precisely, we say that a normal monoid $M\subseteq\ZZ^{\NN}$ is \emph{$\widehat{\mndcl}$-finitely generated} if
\[
 M=\widehat{\mndcl}(A)
 \quad\text{for a finite set }\ A\subset\ZZ^{\NN},
\]
and $M$ is \emph{$\Sym$-equivariantly $\widehat{\mndcl}$-finitely generated} if
\[
 M=\widehat{\mndcl}(\Sym(A))
 \quad\text{for a finite set }\ A\subset\ZZ^{\NN}.
\]
Next consider a $\Sym$-invariant chain of monoids ${\M}=({M}_{n})_{n\geq 1}$ with $M_n\subseteq\ZZ^n$ being normal for all $n\ge1.$ Since
$\mndcl(\Sym(n)(M_m))\subseteq M_n$ and $M_n$ is normal, it holds that
\begin{equation}
 \label{e:stabilization}
 \widehat{\mndcl}(\Sym(n)(M_m))\subseteq M_n
 \quad\text{for all }\ n\ge m\ge 1.
\end{equation}
Hence, ${\M}$ is also a $\Sym$-invariant chain with respect to the operation $\widehat{\mndcl}$. We say that ${\M}$ \emph{$\widehat{\mndcl}$-stabilizes} if the inclusion in \eqref{e:stabilization} becomes an equality for all $n\ge m\gg0$. It is natural to ask whether an $\widehat{\mndcl}$-stabilizing chain of normal monoids also stabilizes (as a chain of monoids, i.e. with respect to the operation $\mndcl$). As we will see in the next section (Corollary~\ref{c:2-stabilization}), this is indeed the case under certain assumptions.

The above properties of normal monoids are related to the corresponding properties of cones as follows.

\begin{proposition}
\label{p:cone-normalmonoid}
 Let $M\subseteq\ZZ^{\NN}$ be a normal monoid and let $C=\cncl(M)$. Then $M=\widehat{\mndcl}(A)$ for some $A\subseteq\ZZ^{\NN}$ if and only if $C=\cncl(A)$. In particular, the following hold:
 \begin{enumerate}
  \item
  $M$ is $\widehat{\mndcl}$-finitely generated if and only if $C$ is finitely generated.
  \item
  $M$ is $\Sym$-equivariantly $\widehat{\mndcl}$-finitely generated if and only if $C$ is $\Sym$-equivariantly finitely generated.
  \item
  Let ${\M}=({M}_{n})_{n\geq 1}$ be a chain of normal monoids $M_n\subseteq\ZZ^n$ and let ${\C}=({C}_{n})_{n\geq 1}$ with $C_n=\cncl(M_n)$ for $n\ge1$. Then ${\M}$ is $\Sym$-invariant if and only if ${\C}$ is $\Sym$-invariant, and when this is the case, ${\M}$ $\widehat{\mndcl}$-stabilizes if and only if ${\C}$ stabilizes.
 \end{enumerate}
\end{proposition}

\begin{proof}
First assume $M=\widehat{\mndcl}(A)$. Since
 \[
  \cncl(A)\subseteq\cncl(M)= \cncl(\widehat{\mndcl}(A))
  \subseteq \cncl(\cncl(A))=\cncl(A),
 \]
equalities must hold throughout and thus $C=\cncl(M)= \cncl(A)$.

Now assume that $C=\cncl(A)$ for some $A\subseteq\ZZ^{\NN}$. Then, as above, it is easily seen that $C=\cncl(\widehat{\mndcl}(A)).$ So using \eqref{eq_normal} we get 
\[
\widehat{\mndcl}(A)=C\cap \ZZ^{\NN}=M,
\]
since $M$ and $\widehat{\mndcl}(A)$ are both normal monoids. 

 Statements (i) and (ii) follow immediately from what we have just shown. For (iii), it is enough to check that $\cncl(\Sym(n)(C_m))=\cncl(\Sym(n)(M_m))$ for $n\ge m\ge 1.$ This results from
 \[
  \cncl(\Sym(n)(C_m))
  =\cncl(\Sym(n)(\cncl(M_m)))
  \subseteq \cncl(\cncl(\Sym(n)(M_m)))
  =\cncl(\Sym(n)(M_m)),
 \]
where the inclusion is due to \cite[Lemma 5.3]{KLR}.
\end{proof}

We conclude this section with the following normal monoid analogue of Corollary~\ref{c:globallocalcones} that is obtained by combining this result and the previous proposition.
\begin{lemma}
	\label{l:globallocalnormalmonoids}
	Let $M\subseteq\ZZ^\infty_{\ge0}$ be a $\Sym$-invariant normal monoid and let $\overline{\M}=(\overline{M}_{n})_{n\geq 1}$ be its saturated chain of local monoids.  Then the following are equivalent:
	\begin{enumerate}
		\item
		$M$ is $\Sym$-equivariantly $\widehat{\mndcl}$-finitely generated;
		\item
		$\overline{\M}$ $\widehat{\mndcl}$-stabilizes and $M_{n}$ is $\widehat{\mndcl}$-finitely generated for all $n\ge 1$;
		\item
		There exists an $s\in\NN$ such that $\overline{M}_{n}$ is $\widehat{\mndcl}$-finitely generated by elements of support size at most $s$ for all $n\ge1$.
	\end{enumerate}
\end{lemma}

 \section{Equivariant Gordan's lemma}
 \label{section-Gordan}

The classical lemma of Gordan states that if $C\subseteq\RR^n$ is a finitely generated rational cone, then $M=C\cap\ZZ^n$ is a finitely generated normal monoid (see, e.g., \cite[Lemma 2.9]{BG}). Here, the rationality of $C$ means that $C$ is generated by elements with rational coordinates. The main result of this section is the following equivariant version of Gordan's lemma.

\begin{theorem}
\label{t:Gordan}
 Let $C\subseteq\RR^{\infty}_{\ge0}$ be a $\Sym$-equivariantly finitely generated rational cone. Then $M=C\cap\ZZ^{\infty}$ is a $\Sym$-equivariantly finitely generated normal monoid.
\end{theorem}

The main idea to prove this result is to pass to the saturated chains $\overline{\C}$ and $\overline{\M}$ of local cones and monoids, apply Gordan's lemma to these local objects, and then go back to the global cone and monoid by using Corollary~\ref{c:globallocalcones} and Lemma~\ref{l:globallocalmonoids}. In order to be able to apply Corollary~\ref{c:globallocalcones} and Lemma~\ref{l:globallocalmonoids}, we first characterize the stabilization of chains of cones and monoids.

\begin{lemma}
\label{l:stabilizing-cone}
  Let $\C=(C_{n})_{n\geq 1}$ be a $\Sym$-invariant chain of cones with $C_n\subseteq\RR^n_{\ge0}$ being finitely generated for all $n\ge1$. Assume further that $\C$ is eventually saturated. Then the following are equivalent:
  \begin{enumerate}
   \item
   $\C$ stabilizes;
   \item
   When $n\gg 0$ and $\ub=(u_1,\dots,u_n)\in C_n$, there exists $\sigma\in \Sym(n)$ such that
  \[
   (u_{\sigma^{-1}(1)},\dots,u_{\sigma^{-1}(n-2)},u_{\sigma^{-1}(n-1)}+u_{\sigma^{-1}(n)})
   \in C_{n-1}.
  \]
  \end{enumerate}
\end{lemma}

To prove this result, we need a simple but useful property of symmetric cones.

\begin{lemma}
 \label{l:between-cone}
 Let $C\subseteq\RR^n$ be a $\Sym(n)$-invariant cone. If $\ub=(u_1,\dots,u_n)\in C$ and $v_{n-1},v_{n}$ lie between $u_{n-1}$ and $u_n$ such that $v_{n-1}+v_{n}=u_{n-1}+u_n$, then $(u_1,\dots,u_{n-2},v_{n-1},v_n)\in C.$
\end{lemma}

\begin{proof}
 Since $C$ is $\Sym(n)$-invariant, $\ub'=(u_1,\dots,u_{n-2},u_{n},u_{n-1})\in C$. By assumption, there exists $\lambda\in[0,1]$ such that
 \[
  v_{n-1}=\lambda u_{n-1}+(1-\lambda)u_n
  \quad\text{and}\quad
  v_{n}=(1-\lambda) u_{n-1}+\lambda u_n.
 \]
Thus, $(u_1,\dots,u_{n-2},v_{n-1},v_n)=\lambda\ub+(1-\lambda)\ub'\in C.$
\end{proof}

This lemma immediately gives the following.

\begin{corollary}
\label{c:between-monoid}
 Let $M\subseteq\ZZ^n$ be a $\Sym(n)$-invariant normal monoid. If $\ub=(u_1,\dots,u_n)\in M$ and $v_{n-1},v_{n}$ are integers lying between $u_{n-1}$ and $u_n$ such that $v_{n-1}+v_{n}=u_{n-1}+u_n$, then $(u_1,\dots,u_{n-2},v_{n-1},v_n)\in M.$
\end{corollary}

Let us now prove Lemma~\ref{l:stabilizing-cone}.

\begin{proof}[Proof of Lemma~\ref{l:stabilizing-cone}]
Since both statements depend only on $C_n$ with $n\gg0$, it is harmless to assume that $\C$ is saturated.

(i)$\Rightarrow$(ii): Suppose $\C$ stabilizes. Then by Corollary~\ref{c:globallocalcones}, we can find a finite uniform bound $s$ for the support sizes of the generators of $C_m$ for all $m\ge 1$. 

Let $n\ge 1$ and $\ub\in C_n$. By Carath\'{e}odory's theorem there exist generators $\vb_1,\dots,\vb_n$ of $C_n$ and $\lambda_1,\dots,\lambda_n\ge0$ such that $\ub=\sum_{i=1}^n\lambda_i \vb_i.$ For $j\in[n]$ set
\[
T_j=\{i\in[n]\colon j\in\supp(\vb_i)\}.
\]
Then double counting yields
\begin{equation}
\label{e:cardinality}
 \sum_{j=1}^n|T_j|=\sum_{i=1}^n|\supp(\vb_i)|\le ns.
\end{equation}
Replacing $\ub$ with $\sigma(\ub)$ for some $\sigma\in \Sym(n)$, if necessary, we may assume that $T_n$ has the smallest cardinality among the $T_j$. Then from \eqref{e:cardinality} we get $|T_n|\le s.$ For simplicity suppose $T_n=\{1,\dots,k\}$ with $k\le s.$ Note that
$\sum_{i=1}^k|\supp(\vb_i)|\le ks$.
So for $n>s^2\ge ks$, there exists
$l\in [n]\setminus\bigcup_{i=1}^k\supp(\vb_i).$ Again by using the action of some $\sigma'\in \Sym(n)$ we may assume that $l=n-1.$ Then $\supp(\vb_i)\subseteq[n]\setminus\{n-1\}$ for $i=1,\dots,k$ and $\supp(\vb_j)\subseteq[n-1]$ for $i=k+1,\dots,n$. Now let $\pi\in \Sym(n)$ be the transposition $(n-1\ n)$ that swaps $n-1$ and $n$. Then $\supp(\pi(\vb_i))\subseteq[n-1]$ for $i=1,\dots,k$. It follows that
\[
 (u_1,\dots,u_{n-2},u_{n-1}+u_n)
 =\sum_{i=1}^k\lambda_i\pi(\vb_i)+\sum_{j=k+1}^n\lambda_j\vb_j\in C_n\cap\RR^{n-1}= C_{n-1},
\]
where the last equality holds because of the assumption that $\C$ is saturated.

(ii)$\Rightarrow$(i): Let $n\gg 0$ and $\ub=(u_1,\dots,u_n)\in C_n$. Then up to an action of some $\sigma\in \Sym(n)$ we have that
$(u_1,\dots,u_{n-2},u_{n-1}+u_n)\in C_{n-1}.$
This implies
$$
(u_1,\dots,u_{n-2},u_{n-1}+u_n,0)\in\cncl(\Sym(n)(C_{n-1})).
$$
Now applying Lemma~\ref{l:between-cone} to the cone $\cncl(\Sym(n)(C_{n-1}))$, one gets $\ub\in\cncl(\Sym(n)(C_{n-1}))$.
Hence, $C_n=\cncl(\Sym(n)(C_{n-1}))$ for all $n\gg 0$, which means that $\C$ stabilizes.
\end{proof}

Lemma~\ref{l:stabilizing-cone} has the following monoid counterpart.

\begin{lemma}
\label{l:stabilizing-monoid}
  Let $\M=(M_{n})_{n\geq 1}$ be a $\Sym$-invariant chain of monoids with $M_n\subseteq\ZZ^n_{\ge0}$ being normal and finitely generated for all $n\ge1$. Assume further that $\M$ is eventually saturated. Then the following are equivalent:
  \begin{enumerate}
   \item
   $\M$ stabilizes;
   \item
   When $n\gg 0$ and $\ub=(u_1,\dots,u_n)\in M_n$, there exists $\sigma\in \Sym(n)$ such that
  \[
   (u_{\sigma^{-1}(1)},\dots,u_{\sigma^{-1}(n-2)},u_{\sigma^{-1}(n-1)}+u_{\sigma^{-1}(n)})
   \in M_{n-1}.
  \]
  \end{enumerate}
\end{lemma}

Before proving this result, let us note the following consequence of Lemma~\ref{l:stabilizing-cone}.

\begin{corollary}
 \label{c:stabilizing-normalmonoid}
 Under the assumptions of Lemma~\ref{l:stabilizing-monoid}, the
 following are equivalent:
  \begin{enumerate}
   \item
   $\M$ $\widehat{\mndcl}$-stabilizes;
   \item
   When $n\gg 0$ and $\ub=(u_1,\dots,u_n)\in M_n$, there exists $\sigma\in \Sym(n)$ such that
  \[
   (u_{\sigma^{-1}(1)},\dots,u_{\sigma^{-1}(n-2)},u_{\sigma^{-1}(n-1)}+u_{\sigma^{-1}(n)})
   \in M_{n-1}.
  \]
  \end{enumerate}
\end{corollary}

\begin{proof}
(i)$\Rightarrow$(ii): This follows easily from Proposition~\ref{p:cone-normalmonoid} and Lemma~\ref{l:stabilizing-cone}.

(ii)$\Rightarrow$(i): One argues almost analogously to the implication (ii)$\Rightarrow$(i) of Lemma~\ref{l:stabilizing-cone}. The only difference is that Lemma~\ref{l:between-cone} is now replaced by Corollary~\ref{c:between-monoid}.
\end{proof}

Now we give the proof of Lemma~\ref{l:stabilizing-monoid}.

\begin{proof}[Proof of Lemma~\ref{l:stabilizing-monoid}]
As in the proof of Lemma~\ref{l:stabilizing-cone}, we may assume that $\M$ is saturated.

(i)$\Rightarrow$(ii): This follows from Corollary~\ref{c:stabilizing-normalmonoid} since $\M$ also $\widehat{\mndcl}$-stabilizes if  it stabilizes.

 (ii)$\Rightarrow$(i): Let $\Hc_n$ denote the Hilbert basis of $M_n$. By Lemma~\ref{l:globallocalmonoids}, it is enough to show that $\|\Hc_n\|\le \|\Hc_{n-1}\|$ for $n\gg 0$. Let $n\gg 0$ and $\ub=(u_1,\dots,u_n)\in \Hc_n$. Then up to an action of some $\sigma\in \Sym(n)$ we have that
 \[
  \ub'=(u_1,\dots,u_{n-2},u_{n-1}+u_n)\in M_{n-1}.
 \]
Since $\|\ub\|=\|\ub'\|$, the desired inequality will follow if we can show that $\ub'\in\Hc_{n-1}.$ If this were not the case, there would exist $\vb'=(v_1,\dots,v_{n-1}),\wb'=(w_1,\dots,w_{n-1})\in M_{n-1}\setminus\{\nub\}$ such that $\ub'=\vb'+\wb'.$ This gives, in particular, that $u_{n-1}+u_n=v_{n-1}+w_{n-1}.$ Without loss of generality we may assume $u_{n-1}\le v_{n-1}.$ Since $\vb'=(v_1,\dots,v_{n-1},0)\in M_{n}$, it follows from Corollary~\ref{c:between-monoid} that $\vb=(v_1,\dots,v_{n-2},u_{n-1},v_{n-1}-u_{n-1})\in M_{n}.$ On the other hand, we see that $\wb=(w_1,\dots,w_{n-2},0,w_{n-1})\in M_{n}$ because $\wb'=(w_1,\dots,w_{n-1},0)\in M_{n}.$ Now from $\ub'=\vb'+\wb'$ we get $\ub=\vb+\wb.$ But this contradicts the assumption that $\ub\in \Hc_n$.
\end{proof}

With Lemmas~\ref{l:stabilizing-cone} and \ref{l:stabilizing-monoid} in hand, we are now ready to prove Theorem~\ref{t:Gordan}.

\begin{proof}[Proof of Theorem~\ref{t:Gordan}]
 Consider the saturated chains $\overline{\C}=(\overline{C}_{n})_{n\geq 1}$ and $\overline{\M}=(\overline{M}_{n})_{n\geq 1}$ of $C$ and $M$, respectively.
By Corollary~\ref{c:globallocalcones}, $\overline{\C}$ stabilizes and $\overline{C}_n$ is a finitely generated rational cone for all $n\ge1$. So by Gordan's lemma, each $\overline{M}_n$ is a finitely generated normal monoid. Now combining Lemmas~\ref{l:stabilizing-cone} and \ref{l:stabilizing-monoid} we deduce that $\overline{\M}$ also stabilizes. Therefore, $M$ is a $\Sym$-equivariantly finitely generated normal monoid by Lemma~\ref{l:globallocalmonoids}.
\end{proof}

Gordan's lemma is known to be false without the rationality assumption of the cone; see, e.g., \cite[Exercise 2.6]{BG}. The next example shows a similar situation in the equivariant setting.

\begin{example}
 Let $\ub=(1,a)\in\RR^2$ with $a>1$ an irrational number. Consider the cone
 \[
  C=\cncl(\Sym(\ub))\subseteq\RR^{\infty}_{\ge0}.
 \]
We show that the monoid $M=C\cap\ZZ^{\infty}$ is not $\Sym$-equivariantly finitely generated. Indeed, let $\overline{\C}=(\overline{C}_{n})_{n\geq 1}$ and $\overline{\M}=(\overline{M}_{n})_{n\geq 1}$ be saturated chains of $C$ and $M$. By Lemma~\ref{l:globallocalmonoids}, it is enough to prove that $\overline{M}_n$ is not finitely generated for all $n\ge 2.$ Assume on the contrary that there exists $n\ge 2$ such that $\overline{M}_n$ is finitely generated, say, by $\vb_i=(v_{i1},\dots,v_{in})$ for $i=1,\dots,m$. Observe that $\overline{\C}$ is exactly the chain examined in Example~\ref{e:1-generator}. Since $\vb_i\in \overline{C}_n$, it follows from Example~\ref{e:1-generator} that
\[
 a(v_{i1}+\dots+v_{i,n-1})-v_{in}>0
 \quad\text{for }\ i=1,\dots,m,
\]
where the inequalities are strict because $a\not\in\QQ$ and $v_{ij}\in\ZZ.$ Set
\[
 \delta=\min\{a(v_{i1}+\dots+v_{i,n-1})-v_{in}\colon 1\le i\le m\}>0.
\]
Obviously, one can choose $\ub=(u_1,\dots,u_n)\in\ZZ^n_{\ge0}$ with $u_1\le\cdots\le u_n$ such that
\begin{equation}
\label{eq:delta}
 0<a(u_1+\cdots+u_{n-1})-u_n<\delta.
\end{equation}
By Example~\ref{e:1-generator}, this implies $\ub\in \overline{C}_n$. Hence $\ub\in \overline{M}_n.$ Now using \eqref{eq:delta} one checks easily that $\ub$ does not belong to the monoid generated by $\vb_1,\dots,\vb_m$, a contradiction.
\end{example}

Finally, let us derive some easy consequences of Theorem~\ref{t:Gordan} and its proof. First of all, we have the following local version of Theorem~\ref{t:Gordan}, which is essentially an immediate step in the proof of this theorem.

\begin{corollary}
 Let $\C=(C_n)_{n\geq 1}$ be a stabilizing $\Sym$-invariant chain of finitely generated rational cones with $C_n\subseteq\RR^n_{\ge0}$ for $n\ge1$. Then
 $\M=(M_{n})_{n\geq 1}$ with $M_n=C_n\cap\ZZ^n$ for $n\ge1$ is a stabilizing $\Sym$-invariant chain of finitely generated normal monoids.
\end{corollary}

\begin{proof}
 First, Gordan's lemma implies that each normal monoid $M_n$ is finitely generated. Let $C=\bigcup_{n\geq 1} C_n$ and $M=\bigcup_{n\geq 1} M_n$. Then $M=C\cap\ZZ^\infty.$ Applying Lemma~\ref{l:globallocalcones} we know that $C$ is a $\Sym$-equivariantly finitely generated rational cone. 
 Hence, $M$ a is $\Sym$-equivariantly finitely generated monoid by virtue of Theorem~\ref{t:Gordan}. The desired conclusion thus follows from \cite[Corollary 5.13]{KLR}.
\end{proof}

Obviously, a stabilizing $\Sym$-invariant chain of normal monoids also $\widehat{\mndcl}$-stabilizes. The next corollary gives the condition under which the converse of this also holds.

\begin{corollary}
 \label{c:2-stabilization}
 Let $\M=(M_{n})_{n\geq 1}$ be a $\Sym$-invariant chain of finitely generated normal monoids with $M_n\subseteq\ZZ^n_{\ge0}$ for $n\ge1$. Then the following are equivalent:
 \begin{enumerate}
   \item
   $\M$ stabilizes;
   \item
   $\M$ $\widehat{\mndcl}$-stabilizes.
 \end{enumerate}
\end{corollary}

\begin{proof}
 We prove (ii)$\Rightarrow$(i). In order to apply Lemma~\ref{l:stabilizing-monoid} and Corollary~\ref{c:stabilizing-normalmonoid}, it suffices to show that $\M$ is eventually saturated. Let $\C=(C_n)_{n\geq 1}$ with $C_n=\cncl(M_n)$ for $n\ge1.$ Then by Proposition~\ref{p:cone-normalmonoid}, $\C$ is a stabilizing $\Sym$-invariant chain of finitely generated cones. So from Lemma~\ref{l:globallocalcones} it follows that $\C$ is eventually saturated. Since $M_n=C_n\cap\ZZ^n$ for $n\ge1$, the chain $\M$ is also eventually saturated.
\end{proof}

Combining Proposition~\ref{p:cone-normalmonoid} and Theorem~\ref{t:Gordan} we immediately obtain the following result, which partially extends \cite[Corollary 2.10]{BG} to the equivariant setting.

\begin{corollary}
\label{cor_normal_monoid}
 Let $M\subseteq\ZZ^\infty_{\ge0}$ be a $\Sym$-invariant normal monoid and let $C=\cncl(M)$. Then the following are equivalent:
 \begin{enumerate}
  \item
  $M$ is $\Sym$-equivariantly finitely generated;
  \item
  $M$ is $\Sym$-equivariantly $\widehat{\mndcl}$-finitely generated;
  \item
  $C$ is $\Sym$-equivariantly finitely generated.
 \end{enumerate}
\end{corollary}

In view of \cite[Corollary 2.10]{BG} one might ask whether the normality assumption on the monoid $M$ can be omitted in the previous result. As the next example shows, this assumption is essential in the equivariant setting. 

\begin{example}
\label{ex:infinite-monoid}
Let $\ub_1=2\eb_1$ and $\ub_k=2\eb_1+\sum_{i=2}^k\eb_i$ for $k\ge 2$. Consider the $\Sym$-invariant monoid $M=\mndcl(\Sym(A))$ with $A=\{\ub_k:k\ge1\}.$
 Let $C=\cncl(M)$. Since $\eb_1\in C$, it is clear that $C=\RR^\infty_{\ge0}=\cncl(\Sym(\eb_1))$. Thus, $C$ is $\Sym$-equivariantly finitely generated. However, $M$ is not $\Sym$-equivariantly finitely generated.
 To see this, let $\overline{\M}=(\overline{M}_{n})_{n\geq 1}$ be the saturated chain of $M$. Then
 $\overline{M}_{n}=\mndcl(\Sym(n)(A_n))$ with $A_n=\{\ub_k:1\le k\le n\}$. 
 We show that $\Hc_n=\Sym(n)(A_n)$ is the Hilbert basis of $\overline{M}_{n}$ for all $n\ge1$. Indeed, it suffices to check that $\ub_k$ is an irreducible element of $\overline{M}_{n}$ for $k=1,\dots,n$. If this were not the case, there would exist $\vb,\wb\in\overline{M}_{n}\setminus\{\nub\}$ such that $\ub_k=\vb+\wb.$ Since $\overline{M}_{n}$ is generated by $\Hc_n$, both $\vb$ and $\wb$ have a coordinate at least 2. Hence, $\vb+\wb$ has either one coordinate at least 4, or two coordinates at least 2. So the equality $\ub_k=\vb+\wb$ cannot occur, confirming that $\Hc_n$ is the Hilbert basis of $\overline{M}_{n}$. Since $\|\Hc_n\|=n+1$ for $n\ge1$, it follows from Lemma~\ref{l:globallocalmonoids} that $M$ is not $\Sym$-equivariantly finitely generated.
\end{example}

\section{Concluding remarks}
\label{sec-Discussion}

In this section let us briefly discuss the connection of our results to algebraic statistics. See \cite[Chapter 9]{Su} for any unexplained terminology. 

Let $\Delta$ be a simplicial complex on $[m]$ and $\rb=(r_1,\dots,r_m)\in\NN^m.$ Let $A_{\Delta,\rb}$ denote the design matrix of the hierarchical model associated with $\Delta$ and $\rb$. An important result in algebraic statistics is the finiteness theorem of Ho\c{s}ten and Sullivant \cite{HoS}, saying that the Markov bases for the lattice $\ker_\ZZ A_{\Delta,\rb}$ eventually stabilize up to symmetry if one of the entries of $\rb$ is allowed to tend to $\infty$, while the other entries are fixed. This generalizes previous finiteness results of Aoki and Takemura \cite{AT} for the no three-way interaction model and of Santos and Sturmfels \cite{SSt} for logit models. Ho\c{s}ten and Sullivant also proposed a far reaching generalization of their result: they conjectured that the finiteness up to symmetry of Markov bases for $\ker_\ZZ A_{\Delta,\rb}$ continues to hold if the entries of $\rb$ with indices in an independent set $T$ are allowed to tend to $\infty$, while the other ones are fixed. Here a subset $T\subseteq [m]$ is called \emph{independent} if it has at most one element in common with any face of $\Delta$. This conjecture, known as the independent set conjecture, was later proven by Hillar and Sullivant \cite{HS12} via an algebraic approach, employing their celebrated result on the equivariant Noetherianity of the polynomial ring $K[X_{[c]\times\NN}]$.

In our terms, the finiteness results of Ho\c{s}ten-Sullivant \cite{HoS} and Hillar-Sullivant \cite{HS12} are \emph{local}, in the sense that they concern chains of lattices of finite rank. It is of interest to know whether there are \emph{global} versions of these results. Addressing this question requires developing a theory of Markov bases for lattices in $\ZZ^\infty$, and more generally, in the free $\ZZ$-module $\ZZ[\NN^k]$ with basis $\NN^k$ for some $k\ge 1.$ In the upcoming paper \cite{LR22}, we study Markov bases and Graver bases for lattices in $\ZZ[\NN^k]$, relying on the framework of \cite{KLR} and the present paper. For any $\Sym$-invariant lattice $L\subseteq\ZZ[\NN^k]$, we relate the finiteness of equivariant Graver bases for $L$ to the finiteness of equivariant Hilbert bases for a certain nonnegative monoid.
When $k=1$, it follows from a result analogous to Theorem~\ref{t:Gordan} that such a monoid always has a finite equivariant Hilbert basis. Thus every $\Sym$-invariant lattice in $\ZZ^\infty$ has a finite equivariant Graver basis. Since the Graver basis contains a Markov basis, this yields (an improvement of) the global version of the finiteness result of Ho\c{s}ten and Sullivant \cite{HoS}. 
It is still open whether the global version of the independent set theorem of Hillar and Sullivant \cite{HS12} can be proven within our framework. We propose in \cite{LR22} a conjecture that implies this result.

\bigskip
{\bf Acknowledgement.} We would like to thank Thomas Kahle for inspiring discussions related to this paper. We are also grateful to the anonymous referees for their insightful and constructive suggestions, especially for pointing out a mistake in the original statement of Corollary~\ref{cor_normal_monoid} that led to Example~\ref{ex:infinite-monoid}.


\begin{thebibliography}{99}

\bibitem{An}
L. Ananiadi,
\emph{Symmetry in toric geometry}.
PhD thesis, Otto-von-Guericke-Universit\"{a}t Magdeburg, 2020.

\bibitem{AT}
S. Aoki and A. Takemura, 
\emph{Minimal basis for a connected Markov chain over $3\times3\times K$ contingency tables with fixed two-dimensional marginals}.
Aust. N. Z. J. Stat. {\bf 45} (2003), no. 2, 229--249.

\bibitem{AH07}
M. Aschenbrenner and C.J. Hillar,
\emph{Finite generation of symmetric ideals}.
Trans. Amer. Math. Soc. {\bf 359} (2007), no. 11, 5171--5192.

\bibitem{BG}
W. Bruns and J. Gubeladze,
\emph{Polytopes, rings, and $K$-theory}.
Springer Monographs in Mathematics, Springer, 2009.

\bibitem{Co67}
D.E. Cohen,
\emph{On the laws of a metabelian variety}.
J. Algebra {\bf 5} (1967), 267--273.

\bibitem{Co87}
D.E. Cohen,
\emph{Closure relations, Buchberger's algorithm, and polynomials in infinitely many variables}.
In: \emph{Computation theory and logic}, 78--87, Lecture Notes in Comput. Sci., {\bf 270}, Springer, Berlin, 1987.


\bibitem{Dr14}
J. Draisma,
\emph{Noetherianity up to symmetry}.
In: {Combinatorial algebraic geometry},
Lecture Notes in Mathematics {\bf 2108}, pp. 33--61, Springer, 2014.

\bibitem{DEF}
J. Draisma,  R.H. Eggermont, and A. Farooq,
\emph{Components of symmetric wide-matrix varieties}.
 J. Reine Angew. Math., 2022, DOI: 
\href{https://doi.org/10.1515/crelle-2022-0064}{10.1515/crelle-2022-0064}.

\bibitem{HKL}
C.J. Hillar, R. Krone, and A. Leykin,
\emph{Equivariant Gr\"{o}bner bases}.
In: \emph{The 50th anniversary of Gr\"{o}bner bases}, 129--154,
Adv. Stud. Pure Math. {\bf 77}, Math. Soc. Japan, Tokyo, 2018.

\bibitem{HS12}
C.J. Hillar and S. Sullivant,
\emph{Finite Gr\"{o}bner bases in infinite dimensional polynomial rings and applications}.
Adv. Math. {\bf 229} (2012), no. 1, 1--25.

\bibitem{HoS}
S. Ho\c{s}ten and S. Sullivant,
\emph{A finiteness theorem for Markov bases of hierarchical models}.
J. Combin. Theory Ser. A {\bf 114} (2007),  311--321.




\bibitem{KLR}
T. Kahle, D.V. Le, and T. R\"{o}mer,
\emph{Invariant chains in algebra and discrete geometry}.
SIAM J. Discrete Math. {\bf 36} (2022), no. 2, 975--999. 

\bibitem{KLS}
R. Krone, A. Leykin, and A. Snowden,
\emph{Hilbert series of symmetric ideals in infinite polynomial rings via formal languages}.
J. Algebra {\bf 485} (2017), 353--362.

\bibitem{LNNR}
D.V. Le, U. Nagel, H.D. Nguyen, and T. R\"{o}mer,
\emph{Castelnuovo-Mumford regularity up to symmetry}.
Int. Math. Res. Not. {\bf 14} (2021), 11010--11049.

\bibitem{LNNR2}
D.V. Le, U. Nagel, H.D. Nguyen, and T. R\"{o}mer,
\emph{Codimension and projective dimension up to symmetry}.
Math. Nachr. {\bf 293} (2) (2020), 346--362.

\bibitem{LR20}
D.V. Le and T. R\"{o}mer,
\emph{A Kruskal-Katona type theorem and applications}.
Discrete Math. {\bf 343} (5) (2020).

\bibitem{LR22}
D.V. Le and T. R\"{o}mer,
\emph{Equivariant lattice bases}.
In preparation.

\bibitem{MR20}
S. Murai and C. Raicu,
\emph{An equivariant Hochster's formula for $\mathfrak{S}_n$-invariant monomial ideals}.
J. London Math. Soc. {\bf 105} (2022), 1974--2010.

\bibitem{Na}
U. Nagel,
\emph{Rationality of equivariant Hilbert series and asymptotic properties},
Trans. Amer. Math. Soc.  {\bf 374} (2021), 7313--7357.

\bibitem{NR17}
U. Nagel and T. R\"{o}mer,
\emph{Equivariant Hilbert series in non-Noetherian polynomial rings}.
J. Algebra {\bf 486} (2017), 204--245.

\bibitem{NR19}
U. Nagel and T. R\"{o}mer,
\emph{{\rm FI}- and {\rm OI}-modules with varying coefficients}.
J. Algebra {\bf 535} (2019), 286--322.

\bibitem{NS}
R. Nagpal and A.~Snowden,
\emph{Symmetric subvarieties of infinite affine space}.
Preprint, 2020, available at \href{https://arxiv.org/abs/2011.09009}
{arXiv:2011.09009}.

\bibitem{SS16}
  S.~Sam and A.~Snowden,
  \emph{$GL$-equivariant modules over polynomial rings in infinitely
    many variables}.
Trans. Amer. Math. Soc. {\bf 368} (2016), no.~2, 1097--1158.

\bibitem{SaSn17}
S.~Sam and A.~Snowden,
\emph{Gr\"obner methods for representations of combinatorial categories}.
J. Amer. Math. Soc. {\bf 30} (2017), no. 1, 159--203.

\bibitem{SSt}
F. Santos and B. Sturmfels, 
\emph{Higher Lawrence configurations}.
J. Combin. Theory Ser. A {\bf 103} (2003), no. 1, 151--164.

\bibitem{Su}
S. Sullivant,
\emph{Algebraic statistics}.
Graduate Studies in Mathematics {\bf 194},
American Mathematical Society, Providence, RI, 2018. 
\end{thebibliography}
\end{document}